\documentclass{amsart}
\usepackage{amssymb, amsthm, amsmath, amsfonts,amscd,bm,fancyhdr}
\usepackage{graphics}
\usepackage{hyperref}
\usepackage[all]{xy}
\usepackage{enumerate}
\usepackage[mathscr]{eucal}
\usepackage{bbm}
\usepackage{tikz}
\usetikzlibrary{decorations.pathreplacing,angles,quotes,knots}

\providecommand{\U}[1]{\protect\rule{.1in}{.1in}}
\def\theenumi{\arabic{enumi}}

\def\theenumii{\alph{enumii}}
\def\p@enumii{\theenumi.}

\def\theenumiii{\arabic{enumiii}}
\def\p@enumiii{(\theenumi)(\theenumii)}

\def\p@enumiv{\p@enumiii.\theenumiii}
\theoremstyle{plain}
\newtheorem{theorem}{Theorem}[section]

\newtheorem{proposition}[theorem]{Proposition}

\newtheorem{corollary}[theorem]{Corollary}
\newtheorem{conjecture}[theorem]{Conjecture}
\numberwithin{equation}{section}

\theoremstyle{definition}

\newtheorem{definition}[theorem]{Definition}
\newtheorem{example}[theorem]{Example}
\newtheorem{remark}[theorem]{Remark}

\newtheorem{thmab}{Theorem}

\renewenvironment{proof}[1][\proofname]{{\bfseries #1\\}}{\qed}

\DeclareMathOperator{\FI}{FI}
\DeclareMathOperator{\VIC}{VIC}

\DeclareMathOperator{\Mod}{-Mod}

\DeclareMathOperator{\Hom}{Hom}
\DeclareMathOperator{\Inn}{Inn}

\DeclareMathOperator{\Aut}{Aut}

\DeclareMathOperator{\Conf}{Conf}

\DeclareMathOperator{\im}{im}
\newcommand{\Sn}{\mathfrak{S}}
\newcommand{\as}{\text{*}}

\newcommand{\C}{{\mathcal{C}}}
\newcommand{\Ob}{\mathcal{O}}

\newcommand{\Z}{{\mathbb{Z}}}

\newcommand{\R}{\mathbb{R}}
\newcommand{\Q}{\mathbb{Q}}
\newcommand{\F}{\mathbb{F}}
\newcommand{\M}{\mathcal{M}}
\newcommand{\K}{\mathcal{K}}
\newcommand{\Par}{\mathcal{P}}
\newcommand{\mi}{\mathfrak{m}}
\newcommand{\ai}{\mathfrak{a}}

\newcommand{\dt}{\bullet}
\newcommand{\arXiv}[1]{\href{http://arxiv.org/abs/#1}{\nolinkurl{arXiv:#1}}}
\newcommand{\arXivV}[2]{\href{http://arxiv.org/abs/#1}{\nolinkurl{arXiv:#1v#2}}}

\title{Asymptotic behaviors in the homology of symmetric group and finite general linear group quandles}

\author{Eric Ramos}
\address{Department of Mathematics, University of Wisconsin - Madison.}
\email{eramos@math.wisc.edu}

\thanks{The author was supported by NSF grants DMS-1502553 and DMS-1704811.}

\begin{document}

\begin{abstract}
A quandle is an algebraic structure which attempts to generalize group conjugation. These structures have been studied extensively due to their connections with knot theory, algebraic combinatorics, and other fields. In this work, we approach the study of quandles from the perspective of the representation theory of categories. Namely, we look at collections of conjugacy classes of the symmetric groups and the finite general linear groups, and prove that they carry the structure of $\FI$-quandles (resp. $\VIC(q)$-quandles). As applications, we prove statements about the homology of these quandles, and construct $\FI$-module and $\VIC(q)$-module invariants of links.
\end{abstract}

\maketitle

\section{Introduction}
A \textbf{quandle} is a set $X$ paired with a binary operation $\rhd$ satisfying the following:
\begin{enumerate}
\item $x \rhd x = x$ for all $x \in X$;
\item $y \mapsto y \rhd x$ is a bijection for all $x \in X$;
\item $(x \rhd y) \rhd z = (x \rhd z) \rhd (y \rhd z)$ for all $x,y,z \in X$.\\
\end{enumerate}
For instance, if $G$ is any group, then $G$ becomes a quandle with operation $x \rhd y = yxy^{-1}$. While group conjugation may be the most obvious, and perhaps the most motivating, example of a quandle, these objects have been shown to appear all throughout algebra and topology. For instance, one can find applications of quandles to knot theory \cite{J,EN,CESY,CJKLS,CKS}, algebraic geometry \cite{T}, and algebraic combintorics \cite{EG}. In \cite{CJKLS}, a theory of quandle homology was introduced, building off previous work of Fenn, Rourke, and Sanderson \cite{FRS}. Since then, there has been a large amount of interest directed towards proving facts about these homology groups. The purpose of this paper is to study quandles and their homology from a new perspective: that of asymptotic algebra.\\

Let $\FI$ denote the category whose objects are the sets $[n] = \{1,\ldots, n\}$ and whose morphisms are injections. In their seminal work \cite{CEF}, Church, Ellenberg, and Farb introduced the notion of an $\FI$-module. It was shown that these modules have a plethora of applications to topology, arithmetic, and algebraic combinatorics. An $\FI$-module over a commutative ring $k$ is a functor from the $\FI$ to the category of $k$-modules. Put another way, an $\FI$-module is an object (in an abelian category) which encodes an infinite family of $\Sn_n$-representations, where $n$ is allowed to vary. Finite generation of an $\FI$-module is then shown to imply remarkably strong facts about the symmetric group representations which constitute it (see Defintion \ref{fg} and Theorem \ref{maincor} or \cite{CEF,CEFN}, for example). These results stress the following philosophy, which we will use in this work: Given a family of algebraic objects which display some kind of asymptotically regular behavior, there is a single object, finitely generated in some abelian category, which encodes the entire family.\\

To begin to state the results of this work, we start with the symmetric group. Recall that conjugacy classes of the symmetric group $\Sn_n$ are in natural bijection with partitions of $n$ (see Definition \ref{primitive}). Let $\lambda$ be a partition of $m$ which does not have any 1's (see Definition \ref{primitive}), and let $c_\lambda$ be the corresponding conjugacy class. Then for each $n \geq m$, we can define $c^n_\lambda$ as the conjugacy class of $\Sn_n$ obtained from $c_\lambda$ by adding $n-m$ 1-cycles. We begin with the following.\\

\begin{thmab}
Let $\lambda$ be a partition of a fixed integer $m$. Then the assignment
\[
n \mapsto c^n_{\lambda}
\]
can be extended to a functor from the $\FI$ to the category of quandles.
\end{thmab}

As an application of of this theorem, we will be able to prove asymptotic facts about the homology of the quandles $c_\lambda^n$ (see Definition \ref{homology}).\\

\begin{thmab}
Let $\lambda$ be a partition of a fixed integer $m$, let $k$ be a commutative Noetherian ring, and let $i \geq 0$ be an integer. Then the assignment
\[
n \mapsto H^Q_i(c_\lambda^{n};k)
\]
can be extended to a finitely generated $\FI$-module (see Definition \ref{fg}). In particular,
\begin{enumerate}
\item If $k$ is a field, then there exists a polynomial $p_{Q,i}(T) \in \Q[T]$ of degree $\leq (i\cdot m)$ such that for all $n \gg 0$,
\[
p_{Q,i}(n) = \dim_k H^Q_i(c^n_{\lambda};k)
\]
\item For each $n$, let $\ai_{Q,i,n} \subseteq k$ be the ideal generated by non-zero-divisors which annihilate $H^Q_i(c^n_{\lambda};k)$. Then for $n \gg 0$, $\ai_{Q,i,n}$ is independent of $n$. In particular, if $k = \Z$, then there exists an integer $e_{Q,i}$, independent of $n$, such that $e_{Q,i}$ is the exponent of the abelian group $H^Q_i(c^n_{\lambda})$, for $n \gg 0$.\\
\end{enumerate}
\end{thmab}

As a second application, we construct $\FI$-module invariants of links. In \cite{J}, Joyce associates to each oriented link $L$ a quandle $\K(L)$ known as the \textbf{fundamental quandle} of $L$ (see Example \ref{funquandle}). One then is then motivated to construct invariants of the link $L$ by looking at the Hom-sets, $\Hom(\K(L),X)$, where $X$ is quandle. These so-called \textbf{quandle colorings} of $L$ have been studied extensively \cite{CESY,EK,EN}. For instance, it can be shown that the Alexander polynomial of links can be recovered from examining certain quandles \cite{EN}. We will prove the following.\\

\begin{thmab}
Let $L$ be an oriented link, and let $\lambda$ be a partition of some fixed integer $m$. Then there exists a finitely generated $\FI$-module $V^{L,\lambda}$ satisfying,
\[
\dim_\Q V^{L,\lambda}([n]) = |\Hom(\K(L),c_{\lambda}^n)|.
\]
In particular, there exists a polynomial $p_{L,\lambda} \in \Q[T]$ such that for all $n \geq 0$
\[
p_{L,\lambda}(n) = |\Hom(\K(L),c_{\lambda}^n)|
\]
\end{thmab}

Note that the results presented in the body of this work are somewhat stronger than the above. Firstly, our methods will allow us to prove Theorems B and C for any finite union of primitive conjugacy classes, not just for single primitive conjugacy classes. Secondly, we also provide bounds on the so-called \textbf{generating degree} of the functors $V^{L,\lambda}$ (see Definition \ref{fg}), and exhibit that they are \textbf{free} (see Example \ref{free}). To accomplish this, we must use a deep structure theorem of Church, Ellenberg, and Farb \cite{CEF}. See Theorems \ref{symhomfg} and \ref{fiinv} for the exact statements.\\

To conclude the paper, we prove analogs of the above theorems for conjugacy classes of the finite general linear groups $GL_n(q)$. We accomplish this by studying representations of the category $\VIC(q)$ (see Definition \ref{catdef}), which was first introduced by Djament \cite{D} and further explored by Putman and Sam \cite{PS}. These results can be found throughout Section \ref{vic}. We note that the similarities in the statements between these two cases is not a coincidence. It is the belief of the author that there should be a framework in the representation theory of more general ``combinatorial'' categories which unifies all of these results (see Remark \ref{more}).\\

\section*{Acknowledgments}
The author is indebted to Jordan Ellenberg and Jennifer Wilson for many fruitful conversations during the conception of this work. Thanks must also be sent to Steven Sam, who caught a mistake in a previous version of this work. The author would also like to acknowledge the generous support of the National Science Foundation through the grants DMS-1502553 and DMS-1704811.\\

\section{Preliminary notions}

\subsection{Quandle and rack homology}

In this section we spend some time outlining the theory of quandles, racks, and their homology. We will find that these objects not only have interesting internal algebraic properties, but also admit many useful applications to the theory of knots and their generalizations. For a reference on the subject, see \cite{FRS,EN}

\begin{definition}
A \textbf{rack} is a set equipped with a binary operation $(X,\rhd)$ satisfying the following conditions:
\begin{enumerate}
\item for all $x,y,z \in X$, 
\[
(x \rhd y) \rhd z = (x \rhd z) \rhd (y \rhd z);
\]
\item for all $x \in X$, the function $\dt \rhd x$ is a bijection. That is, for all $x,z \in X$ there is a unique $y \in X$ such that $y \rhd x = z$.
\end{enumerate}
A \textbf{quandle} is a rack $(X,\rhd)$ with the added reflexivity condition:
\begin{enumerate}
\setcounter{enumi}{2}
\item for all $x \in X$,
\[
x \rhd x = x.
\]
\end{enumerate}
\end{definition}

In this work, we will be primarily concerned with quandles, although certain results will hold for more general racks. We take some time to exhibit some important examples, which will appear throughout the work.

\begin{example}
Let $G$ be any group. Then $G$ forms a quandle under conjugation. Namely, for any $x,y \in G$,
\[
x \rhd y := yxy^{-1}
\]
More generally, if $X$ is any union of conjugacy classes of a group $G$, then $X$ forms a quandle under conjugation. Indeed, one may think of quandles as an attempt to axiomatize conjugation.\\
\end{example}

\begin{example}\label{dihedral}
Let $r$ be a fixed positive integer. Then $\Z/r\Z$ is a quandle with operation given by
\[
x \rhd y = 2y - x
\]
This quandle is often called the dihedral quandle, and has been studied extensively \cite{NP,Cl,P}. Note that this quandle is equivalent to the conjugacy class of reflections in the dihedral group $D_{2r}$.\\
\end{example}

\begin{example}\label{funquandle}
Let $L$ be an oriented link, and choose a projection of $L$ onto the plane, $D(L)$, keeping track of under and over crossings. Such an object is also known as a \textbf{diagram} for the link $L$. An \textbf{arc} of $D(L)$ is an embedded copy of the interval found between two undercrossings (see Figure \ref{arcs}). Then we may associate a quandle to $D(L)$, usually called the \textbf{fundamental quandle of $L$}, $\K(L)$, by taking the free quandle formally generated by the arcs of $D(L)$ and imposing the relations prescribed by Figure \ref{knotrels}. See Figure \ref{arcs} for an example of the fundamental quandle of the trefoil knot.\\

The fundamental quandle was introduced by Joyce in his work \cite{J}. It is a fact that the fundamental quandle is independent of the choice of diagram. That is, it is an invariant of the link $L$. This can be seen by noting that the axioms of quandles can be translated, using the relations of Figure \ref{knotrels}, into the Reidemeister moves. It was proven by Joyce \cite{J}, that if $L$ is a knot, then its fundamental quandle determines it up to a homeomorphism of $S^3$. We will later use the fundamental quandle to define link invariants.\\
\end{example}

\begin{figure}
\begin{tikzpicture}

	\draw[ultra thick, ->] (0,0) node (1)[above]  {$x$}
			--(2,2) node (2)[above] {$x$};
	\draw[ultra thick]	(2,0) node (3)[above] {$y$}
    	-- (1.2,.8) node (4) {};
	\draw[ultra thick] (.8,1.2) node (5) {}
			--(0,2) node (6)[above] {$y \rhd x$};
		 
\end{tikzpicture}
\caption{The defining relation of the fundamental quandle of a link. One may read a relation of the form $y \rhd x = z$ as ``the arc $y$ goes under the arc $x$ and becomes the arc $z$.''}\label{knotrels}
\end{figure}
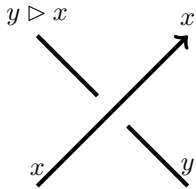

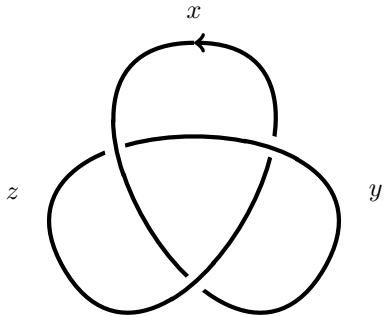
\begin{figure}

\begin{tikzpicture}
\begin{knot}[
consider self intersections,
clip width=5,
flip crossing = 2
]
\strand[ultra thick,->]
(90:2) to[out=180,in=-120,looseness=2]
(-30:2) to[out=60,in=120,looseness=2]
(210:2) to[out=-60,in=0,looseness=2] (90:2);
\end{knot}
\draw (0,2.2) node[above] {$x$};
\draw (-2.2,0) node[left] {$z$};
\draw (2.2,0) node[right] {$y$};
\end{tikzpicture}

\caption{A diagram of the trefoil knot $3_1$. This diagram has three arcs, labeled $x,y,$ and $z$. Using the relation of Figure \ref{knotrels}, we discover that $\K(3_1) = \langle x,y,z \mid y \rhd z = x, z \rhd x = y, x \rhd y = z \rangle$. The LaTex code for the above diagram can be found in \cite{St}}\label{arcs}
\end{figure}

\begin{definition}\label{homology}
Let $(X,\rhd)$ be a rack, and for each $i\geq 1$, let $C_i^R(X)$ be the free $\Z$-module on the elements of $X^i$. for $i > 1$, we define a differential $\partial_i:C_i \rightarrow C_{i-1}$,
\[
\partial_i(x_1,\ldots,x_i) = \sum_{j\geq 2} (-1)^j((x_1,\ldots,x_{j-1},x_{j+1},\ldots,x_i) - (x_1 \rhd x_j,x_2 \rhd x_j, \ldots,x_{j-1}\rhd x_j, x_{j+1},\ldots,x_i)).
\]
By convention we set $\C_0^R(X) = 0$ to be the trivial. The \textbf{$i$-th rack homology of $X$}, $H^R_i(X)$, is defined to be the $i$-th homology of the complex
\[
C_\dt^R:\ldots \rightarrow C_{i+1}^R(X) \stackrel{\partial_{i+1}}\rightarrow C_i^R(X) \rightarrow C_{i-1}^R(X) \rightarrow \ldots
\]
If $X$ is a quandle, then there is a quotient complex of $C_\dt^R$ whose terms are given by
\[
C_i^Q(X) := C_i^R(X)/((x_1,\ldots,x_i) \mid x_j = x_{j+1} \text{ for some $j$})
\]
The \textbf{$i$-th quandle homology}, $H^Q_i(X)$ is the $i$-th homology of this complex. If $A$ is any abelian group we may also define, via the usual universal coefficient theorem, $H^Q_i(X;A)$ and $H^R_i(X;A)$.\\
\end{definition}

Homology groups of quandles and racks have been an active field of study since their discovery in \cite{CJKLS,FRS}. One reason for this is their deep connections with knot theory \cite{CEGS, CJKLS,EN}. Another reason is due to how surprisingly hard these groups are to compute. While the Betti numbers have been classified in many cases (see Theorem \ref{betticomp}), torsion is still not totally understood. One of the major computational difficulties derives from the fact that there is no clear way to interpret quandle homology as the homology of some topological space. There have, however, been many partial results in this direction \cite{Cl,CKS,IK}.\\

\begin{theorem}[Litherland \& Nelson, \cite{LN}; Etingof \& Gra\~na, \cite{EG}]
Let $X$ be a quandle. Then the quotient map
\[
C^R_\dt(X) \rightarrow C^Q_\dt(X)
\]
splits. Moreover, if $X$ is a finite rack, then any primes which annihilate $H_i^R(X)$ are divisors of $|\Inn(X)|$, where
\[
\Inn(X) := \langle\cdot \rhd x \mid x \in X \rangle \leq \Aut_{\text{rack}}(X).
\]
\end{theorem}

Note that the statement about the splitting in the above theorem was proven by Litherland and Nelson \cite{LN}, while the statement on torsion was exhibited by Etingof and Gra\~na \cite{EG}. Litherland and Nelson do also prove a statement about the torsion part of the rack homology groups in \cite{LN}, although they require certain technical conditions be placed on the rack. In particular, they prove that if $X$ is a finite rack which is sufficiently nice, then the exponent of $H_i^R(X)$ is a divisor of $|X|^i$.  Computations suggest that the exponent of $H_i^R(X)$ is often smaller than $|X|^i$. In fact, we will later construct infinitely many families of quandles for which $|X|^i$ is strictly bigger than the actual exponent.\\

One striking fact about quandle and rack homology is that the Betti numbers are very easily computable.\\

\begin{theorem}[Etingof \& Gra\~na, \cite{EG}; Litherland \& Nelson, \cite{LN}]\label{betticomp}
Let $X$ be a finite rack, and write $m$ for the number of orbits of the action of $X$ on itself via right multiplication. Then,
\begin{enumerate}
\item $\dim_\Q(H^R_i(X;\Q)) = m^i$;
\item $\dim_\Q(H^Q_i(X;\Q)) = m(m-1)^{i-1}$ if $X$ is a quandle.\\
\end{enumerate}
\end{theorem}

Note that the above proven was proven in the provided generality by Etingof and Gra\~na in \cite{EG}. Litherland and Nelson had proven the statement for a certain class of racks in \cite{LN}. One benefit of the techniques used in this paper is that they allow us to study the homology groups with very general coefficients.\\

\subsection{The Representation Theory of Categories}

\begin{definition}
Let $\C$ be a (small) category, and let $k$ be a commutative ring. A \textbf{representation of $\C$ over $k$}, or a \textbf{$\C$-module over $k$}, is a covariant functor $V:\C \rightarrow k\Mod$.
\end{definition}

The representation theory of categories has seen a recent boom in the literature (see \cite{CEF,CEFN,EW-G,HR,MW,SS,PS} for a small taste). Much of this traces back to the incredible success of $\FI$-modules, as defined by Church, Ellenberg, and Farb in \cite{CEF}. In this work will be applying the representation theory of $\C$-modules to quandle and rack homology.\\

\begin{definition}\label{catdef}
The category $\FI$ is that whose objects are the finite sets $[n] := \{1,\ldots,n\}$ and whose morphisms are injections. For a fixed finite field $F := \F_q$, we define $\VIC(q)$ to be the category whose objects are the vector spaces $F^n, n = 0,1,\ldots$, and whose morphisms are pairs $(f,W):F^n \rightarrow F^m$ such that $f:F^n \rightarrow F^m$ is a linear injection, and $W \subseteq F^m$ has the property that $f(F^n) \oplus W  = F^m$.
\end{definition}

\begin{remark}
If $V$ is either and $\FI$-module, or a $\VIC(q)$-module, we will use $V_n$ to denote $V([n])$ (resp. $V(F^n)$). Note that $\FI$-modules were introduced by Church, Ellenberg, and Farb in \cite{CEF}, while $\VIC(q)$-modules were introduced by Djament in \cite{D}, and expanded upon greatly by Putman and Sam in \cite{PS}.\\
\end{remark}

One immediately observes that if $\C = \FI$, then $\Aut([n]) = \Sn_n$, the symmetric group on $n$ letters. Similarly, if $\C = \VIC(q)$, then $\Aut(F^n) = GL_n(q)$. This therefore suggests the following interpretation of representations of these categories. We imagine a $\C$-modules as sequence of $\Sn_n$, or $GL_n(q)$, representations, with $n$ varying. These representations are then given some kind of compatibility through the actions of the maps induced from the morphisms of $\C$.\\

\begin{remark}\label{more}
Looking through the literature, one will note that there are many more categories other than $\FI$ and $\VIC(q)$ which have been studied. For instance, if $G$ is a finite group then one may consider $\FI_G$-modules (see \cite{SS2,W,Ca,G}, for a few examples). In this case, the acting groups are the wreath product $\Sn_n \wr G$. One may also consider the category $\FI^m$, whose acting groups are products of symmetric groups $\Sn_{n_1} \times \ldots \times \Sn_{n_m}$ \cite{G,LY}. We therefore stress the following: \emph{The analyses and results which will be discussed in Sections \ref{fi} and \ref{vic} should have analogs for many other important categories}. For sake of brevity, as well as to avoid repeating arguments, we will only work with $\FI$ and $\VIC(q)$-modules. We hope that this work can be used as inspiration for applying certain arguments to the study of quandles and quandle homology.\\
\end{remark}

Many of the definitions and statements which follow will make sense for both $\FI$ and $\VIC(q)$-modules. \emph{For the remainder of this work, we will reserve $\C$ to denote either the category $\FI$ or the category $\VIC(q)$.} We will, of course, point out cases where the two categories need to be differentiated.\\

Many natural notions from the study of $k$-modules can be translated to the language of $\C$-modules.\\

\begin{definition}\label{fg}
Let $V$ be a $\C$-module, and let $d \geq 0$ be an integer. The category of $\C$-modules and natural transformations is an abelian category, with the usual abelian operations defined point-wise. A \textbf{submodule} of $V$ is a $\C$-module $W$ along with an injective morphism $W \hookrightarrow V$. We say that $V$ is \textbf{finitely generated in degree $\leq d$} if there is a finite set $\{v_i\} \subseteq \oplus_{n = 0}^d V_n$ which is not contained in any proper submodule of $V$.\\
\end{definition}

For many applications, it is often useful to limit our scope to finitely generated $\C$-modules and their properties. Before we detail these properties, we spend a moment examining some examples.\\

\begin{example}\label{free}
Fix a non-negative integer $r$ and a commutative ring $k$. We define the \textbf{principal projective module generated in degree $r$}, $M(r)$, by setting
\[
M(r)_n = k[\Hom_{\C}(r,n)],
\]
the free $k$-module on the set $\Hom_{\C}(r,n)$. For any morphism $\phi \in \Hom_{\C}(n,m)$, the map $M(r)(f)$ is defined on basis vectors by composition. It can be seen that $M(r)$ is generated in degree $r$ by the basis vector $id_r$.\\

More generally, let $W$ be an $\Sn_r$-representation over $k$ (resp. a $GL_n(q)$-module over $k$). Then we define the $\FI$-module (resp. $VIC(q)$-module), $M(W)$, via the assignments
\[
M(W)_n = M(r)_n \otimes W,
\]
where the tensor product is over $k[\Sn_r]$ (resp. $k[GL_n(q)]$). Direct sums of modules of the form $M(W)$ are known as \textbf{free modules.}\\
\end{example}

\begin{example}
For a more topologically motivated example, let $\M$ denote an oriented manifold of dimension $\geq 2$.  The $n$-strand configuration space on $\M$ is the space
\[
\Conf_n(\M) := \{(x_1,\ldots,x_n) \in \M^n \mid x_i \neq x_j\}.
\]
For any injection of sets $f:[n] \hookrightarrow [m]$, we obtain a continuous map $\Conf_m(\M) \rightarrow \Conf_n(\M)$ given by forgetting points in a way consistent with $f$. For any fixed index $i$, we may compose with the functor $H^i(\dt)$ to obtain an $\FI$-module over $\Z$
\[
H^i(\Conf_\dt(\M)).
\]
It is a theorem of Church \cite{Ch}, which was later expanded upon by Church, Ellenberg, and Farb \cite{CEF}, that the $\FI$-module $H^i(\Conf_\dt(\M))$ is actually finitely generated. We will soon see the plethora of non-trivial facts that this implies about the cohomology groups.\\
\end{example}

\begin{example}
Let $K$ be an algebraic number field (i.e. a finite field extension of $\Q$), and let $\Ob_K$ denote its ring of integers (i.e. the integral closure of $\Z$ in $K$). For any maximal ideal $\mi \subseteq \Ob_K$, the quotient $\Ob_K/\mi$ is a finite field. Number theorists are often concerned with the congruence subgroup,
\[
GL_n(\Ob_K,\mi) := \ker(GL_n(\Ob_K) \rightarrow GL_n(\Ob_K/\mi)).
\]
It was proven by Putman and Sam \cite{PS} that for each fixed index $i$ the collection
\[
H_i(GL_\dt(\Ob_K,\mi))
\]
can be endowed with the structure of a finitely generated $\VIC(|\Ob_K/\mi|)$-module.\\
\end{example}

If $k$ is a commutative ring, then one finds it is often times massively useful to know that $k$-modules satisfy a kind of Noetherian property. Namely, while performing homological computations wherein modules are often constructed as subquotients of other modules, one would like to be able to say something about finite generation. For this purpose, we have the following theorem.\\

\begin{theorem}[Church, Ellenberg, Farb, \& Nagpal, \cite{CEFN}; Putman \& Sam, \cite{PS}]
Let $k$ be a Noetherian ring, and let $V$ be a finitely generated $\C$-module. Then every submodule of $V$ is also finitely generated.\\
\end{theorem}

The above Noetherian property was proven for $\FI$-modules in the case wherein $k$ is a field of characteristic 0 by Snowden \cite{Sn}, and independently by Church, Ellenberg, and Farb \cite{CEF}. It was proven in the above generality for $\FI$-modules by Church, Ellenberg, Farb, and Nagpal in \cite{CEFN}. It was proven for $\VIC(q)$-modules by Putman and Sam in \cite{PS}.\\

Computing many useful properties of finitely generated $\C$-modules depends on computing the generating degree of the module. However, because of the non-constructive nature of the Noetherian property, this isn't always possible to do. If $W$ is a $\C$-module which arises as a subquotient of a module which is finitely generated in degree $\leq d$, then one can often bound invariants of $W$ using $d$, though these bounds might not be optimal. We therefore have the following definition, which we borrow from \cite{PY}.\\

\begin{definition}
We say that a $\C$-module over a commutative ring $k$ is \textbf{$d$-small} if it is a subquotient of a $\C$-module which is finitely generated in degree $\leq d$. Note that while $d$-small modules are finitely generated whenever $k$ is Noetherian, the degree of generation is apriori independent of $d$.\\
\end{definition}

\begin{theorem}\label{maincor}
Let $k$ be a Notherian ring, and let $V$ be a $\C$-module which is $d$-small.
\begin{enumerate}
\item $[$Church, Ellenberg, Farb, \& Nagpal, \cite{CEFN}$]$ If $k$ is a field, and $\C = \FI$, then there exists a polynomial $p_V(T) \in Q[T]$ of degree $\leq d$ such that for all $n \gg 0$,
\[
p_V(n) = \dim_kV_n.
\]
\item $[$Gan \& Watterlond, \cite{GW}$]$ If $k$ is a field of characteristic 0, and $\C = \VIC(q)$, then there exists a polynomial $p_V(T) \in \Q[T]$ of degree $\leq d$ such that for all $n \gg 0$
\[
p_V(q^n) = \dim_k V_n.
\]
\item For each $n$, let $\ai_n \subseteq k$ be the ideal generated by non-zero-divisors which annihilate $V_n$. Then for $n \gg 0$, $\ai_n$ is independent of $n$. In particular, if $k = \Z$, then there exists an integer $e_{V}$, independent of $n$, such that $e_V$ is the exponent of the abelian group $V_n$ for $n \gg 0$.
\end{enumerate}
\end{theorem}

\begin{proof}
It only remains to prove the third statement. For each $n$, let $T_n \subseteq V_n$ be the submodule of elements which are annihilated by some non-zero-divisor $x \in k$. It isn't hard to see that the collection of $T_n$ constitute a submodule of $V_n$. By the Noetherian property for $\C$-modules, we know that $T$ is finitely generated. In this case, we only need to worry about those elements of $k$ which annihilate the (finitely many) generators of $T$.\\
\end{proof}

In the case where $k$ is a field of characteristic 0, Church, Ellenberg, and Farb prove the first statement of the above theorem \cite{CEF}. In that work, they also give bounds on when the claimed polynomial behavior actually starts. Work of the author \cite{R}, as well as the Li and the author \cite{LR}, provide bounds on when the polynomial behavior begins if $k$ is any field. To the knowledge of the author, the second statement of the above theorem has never been made effective.\\

\begin{remark}
It is the belief of the author that the second statement should be true so long as $k$ is a field of characteristic prime to $p$. The second statement is provably false if $k$ is a field of characteristic $p$. Indeed, consider the $\VIC(q)$-module over $F := \F_q$, $V$, given by
\[
V_n = F^n
\]
where the action of the morphisms of $\VIC(q)$ is the obvious one.\\
\end{remark}

\subsection{Quandle colorings of knots and links}

Recall that we can associate to any link a quandle $\K(L)$, called the fundamental quandle of $L$ (see Example \ref{funquandle}). As a consequence, if $X$ is any quandle, we may consider the set
\[
\Hom_{\text{quandle}}(\K(L),X)
\]

\begin{definition}
If $X$ is a quandle and $L$ is a link, then a \textbf{coloring of $L$ by $X$} is a quandle homomorphism $\phi:\K(L) \rightarrow X$. The \textbf{quandle counting invariant of $L$} is the function
\[
\chi_L(X) := |\Hom_{\text{quandle}}(\K(L),X)|
\]
\end{definition}

\begin{example}
Let $L = 3_1$ be the trefoil knot as pictured in Figure \ref{arcs}, and let $X$ be $\Z/3\Z$ be the dihedral quandle of Example \ref{dihedral}. Then a coloring of $3_1$ by $X$ is equivalent to an ordered triple $(x,y,z) \in (\Z/3\Z)^3$ such that $x+y+z = 0 \pmod{3}$. This is often times stated in the following way: a coloring of a knot by the dihedral quandle $\Z/3\Z$ is a way to assign one of three colors to each arc in such a way that at any crossing the three arcs involved are all the same color, or are all of differing colors. We conclude that $\chi_{3_1}(X) = 9$.\\

Colorings of links by the dihedral quandle $\Z/n\Z$ have been studied extensively. See \cite{P} for a survey, including connections of $\chi_{L}(\Z/n\Z)$ with the Jones and Kaufmann polynomial invariants.\\
\end{example}

Quandle colorings have been studied very extensively, as in many cases they provide fairly easily computable invariants. For instance, the classic Alexander Polynomial invariant can be realized by studying certain quandle colorings. For a small sampling of of results in this direction, see \cite{CESY, EK, EN}.\\

The Yoneda lemma implies that knowing all possible quandle colorings of a link $L$ uniquely determines the fundamental quandle of $L$. In the cases wherein $L$ is a knot, the theorem of Joyce \cite{J} further implies that all possible quandle colorings of a knot determine the knot. In fact, it has been conjectured that there is some infinite sequence of finite quandles whose colorings are sufficient to distinguish any two knots \cite{CESY}.

\begin{conjecture}[Clark, Elhamdadi, Saito, \& Yeatman, \cite{CESY}]
There exists a sequence of finite quandles $S = (Q_1, Q_2, \ldots , Q_n, \ldots)$ such that the
invariant
\[
C(L) := (\chi_{L}(Q_1),\chi_{L}(Q_2), \ldots, \chi_{L}(Q_n),\ldots)
\]
satisfies for all knots K, K'
\[
C(K') = C(K) \text{ if and only if } \K(K') = \K(K).
\]
\end{conjecture}

In a way, this conjecture is the inspiration for the knot invariants constructed in the later sections of this work. Namely, if $X$ is a finite quandle, then one expects that the counting invariant $\chi_L(X)$ will only take on a limited list of values as $L$ varies. This suggests that the counting invariant of a finite quandle is not particularly strong. However, one would like to work with finite quandles, as they make computations significantly more tractable. We therefore construct an infinite sequence of finite quandles whose colorings can be embedded into a single object, which is finitely generated in the appropriate sense. In other words, for a given link $L$ we hope to construct finitely generated $\C$-modules which encode some infinite collection of quandle colorings of $L$ by finite quandles.\\

\section{Some $\FI$-quandles}\label{fi}

\subsection{Definition and basic properties}

To begin, we spend a moment recalling some facts about conjugacy classes of the symmetric groups $\Sn_m$.\\

\begin{definition}\label{primitive}
Let $\lambda$ be a partition of $m$. Then $\lambda$ corresponds to a conjugacy class $c_\lambda$ of $\Sn_m$. Under this correspondence, an element of $c_\lambda$ has cycle structure given by
\[
\text{\# of $i$-cycles} = |\{j \mid \lambda_j = i\}|.
\]
We say that a conjugacy class $c_\lambda$ is \textbf{primitive} if $\lambda_i \neq 1$ for all $i$. We also say a partition $\lambda$ is \textbf{primitive} under the same circumstances.\\

Given a partition $\lambda$ of $m$ and $r \geq 0$, define $\lambda^r$ to be the partition of $m+r$
\[
\lambda^r := (\lambda,\underbrace{1,1,\ldots,1}_{r \text{ times}})
\]
If $c_\lambda$ is a primitive conjugacy class of $\Sn_m$ and $n \geq m$, we write $c^n_\lambda$ for $c_{\lambda^{n-m}}$. If $n < m$, then we set $c^n_\lambda = \emptyset$.\\
\end{definition}

Given a group $G$, any union of conjugacy classes of $G$ carries the structure of a quandle under conjugation. In the case of the symmetric groups, if $c_\lambda$ is a primitive conjugacy class of $\Sn_m$, then one may consider $c^n_{\lambda}$, with $n \geq m$, as the ``same'' conjugacy class, just transported to $\Sn_n$. This sameness is the basis for many of the finite generation theorems which will follow.\\

\begin{definition}
Let $\Par$ be the set integer partitions (not necessarily of the same integer), and let $\bm{\lambda} \subseteq \Par$ be a finite set of primitive partitions. For each non-negative integer $n$ we define the quandle
\[
c^n_{\bm{\lambda}} := \cup_{\lambda \in \bm{\lambda}}c^n_{\lambda} \subseteq \Sn_n
\]
If $f:[m] \hookrightarrow [n]$ is an injection, then we obtain a map of quandles $f_\as:c^m_{\bm{\lambda}} \rightarrow c^n_{\bm{\lambda}}$ via the assignment
\[
f_\as(\sigma)(i) = \begin{cases} f\circ \sigma \circ f^{-1}(i) & \text{ if $i \in \im(f)$}\\ i & \text{ otherwise.}\end{cases}
\]
With the above, the assignment $n \mapsto c^n_{\bm{\lambda}}$ defines a functor $c^\dt_{\bm{\lambda}}$ from $\FI$ to the category of quandles.
\end{definition}

In their work \cite{CESY}, Clark, Elhamdadi, Saito, and Yeatman construct a list of 26 finite quandles whose colorings can distinguish all prime (i.e. not expressible as a connect-sum of non-trivial knots) knots with less than 12 crossings, up to certain symmetries. Some quandles on this list are of the form $c^n_{\bm{\lambda}}$ for some choice of $n$ and $\bm{\lambda}$, although this language is not used in that work. It has also been shown \cite{Cl2,HMN} that conjugation quandles of the symmetric groups are useful in classifying finite quandles.\\

\subsection{Proving that the homology is finitely generated}

In this section we focus on proving that for any choice of $\bm{\lambda}$, any fixed index $i$, and any commutative Noetherian ring $k$, the homology groups $H^R_i(c^\dt_{\bm{\lambda}};k)$ and $H^Q_i(c^\dt_{\bm{\lambda}};k)$ form finitely generated $\FI$-modules. First, we need some notation. If $\lambda$ is a partition of $m$, then we write $|\lambda| = m$. If $\bm{\lambda} \subseteq \Par$ is a finite union of primitive partitions, then we write 
\[
|\bm{\lambda}| = \max_{\lambda \in \bm{\lambda}}\{|\lambda|\}
\]

\begin{theorem}\label{symhomfg}
Let $\bm{\lambda} \subseteq \Par$ be a finite set of primitive partitions, let $i \geq 0$ be a fixed integer, and let $k$ be a commutative Noetherian ring. Then the $\FI$-modules $H^R_i(c^\dt_{\bm{\lambda}};k)$ and $H^Q_i(c^\dt_{\bm{\lambda}};k)$ are finitely generated. Moreover, both $H^R_i(c^\dt_{\bm{\lambda}};k)$ and $H^Q_i(c^\dt_{\bm{\lambda}};k)$ are $(i\cdot |\bm{\lambda}|)$-small.
\end{theorem}

\begin{proof}
We will first prove the statement for the rack homology groups. For each $i \geq 1$, recall that if $X$ is a quandle we set $C_i(X;k)$ to be the free $k$-module on the set $X^i$ (and that $C_0(X,k) = 0$). We claim that $C_i(c^\dt_{\bm{\lambda}};k)$ is a finitely generated $\FI$-module, generated in degree $\leq i\cdot |\bm{\lambda}|$.\\

First observe that $(c^\dt_{\bm{\lambda}})^i$ inherits the structure of an $\FI$-quandle. By extending this action $k$-linearly, we obtain the natural structure of an $\FI$-module on $C_i(c^\dt_{\bm{\lambda}};k)$ to show that this $\FI$-module is finitely generated in degree $\leq i\cdot |\bm{\lambda}|$ it remains to show that if $n >i\cdot |\bm{\lambda}|$, then every basis element of $C_i(c^n_{\bm{\lambda}};k)$ arises as the image of a basis element of $C_i(c^{n-1}_{\bm{\lambda}};k)$ under the action of $\FI$. Let $(\sigma_1,\ldots,\sigma_i) \in C_i(c^n_{\bm{\lambda}};k)$. Observe that at most $i\cdot|\bm{\lambda}|$ total elements of $[n]$ are not fixed by all the $\sigma_j$. In particular, if $n > i \cdot |\bm{\lambda}|$, then there exists at least one element $l \in [l]$ which is fixed by all of the $\sigma_j$. Let $f:[n-1] \hookrightarrow [n]$ be the function defined by
\[
f(j) = \begin{cases} j &\text{ if $j < l$}\\ j+1 &\text{ if $j \geq l.$}\end{cases}
\]
Then $(\sigma_1,\ldots,\sigma_i)$ is the image of $(\tau_1,\ldots,\tau_i)$ under $f_\as$, where
\[
\tau_j = f^{-1} \circ \sigma_j \circ f.
\]
We next note that the differential $\partial$, commutes with the action of $\FI$. That is to say, the complex
\[
C_{\star}(c^\dt_{\bm{\lambda}};k): \ldots \rightarrow C_i(c^\dt_{\bm{\lambda}};k) \rightarrow C_{i-1}(c^\dt_{\bm{\lambda}};k) \rightarrow \ldots
\]
is actually a complex of $\FI$-modules. The Noetherian property now implies that the homology of this complex, $H_i(c^\dt_{\bm{\lambda}};k)$, is finitely generated as an $\FI$-module. The above computation also implies that the homology is $(i\cdot |\bm{\lambda}|)$-small.\\

The proof of the analogous statement for quandle homology is identical.\\
\end{proof}

Combining Theorem \ref{symhomfg} with Theorem \ref{maincor} we immediately obtain the following.\\

\begin{corollary}\label{symcor}
Let $\bm{\lambda} \subseteq \Par$ be a finite set of primitive partitions, let $i \geq 0$ be a fixed integer, and let $k$ be a Noetherian ring.
\begin{enumerate}
\item If $k$ is a field, then there exist polynomials $p_{R,i}(T),p_{Q,i}(T) \in Q[T]$ of degree $\leq i \cdot |\bm{\lambda}|$ such that for all $n \gg 0$,
\[
p_{R,i}(n) = \dim_k H^R_i(c^n_{\bm{\lambda}};k), \text{ and } p_{Q,i}(n) = \dim_k H^Q_i(c^n_{\bm{\lambda}};k)
\]
\item For each $n$, let $\ai_{R,i,n},\ai_{Q,i,n} \subseteq k$ be the ideals generated by non-zero-divisors which annihilate $H^R_i(c^n_{\bm{\lambda}};k)$ and $H^Q_i(c^n_{\bm{\lambda}};k)$, respectively. Then for $n \gg 0$, $\ai_{R,i,n}$ and $\ai_{Q,i,n}$ are independent of $n$. In particular, if $k = \Z$, then there exist integers $e_{R,i}$ and $e_{Q,i}$, independent of $n$, such that $e_{R,i}$ and $e_{Q,i}$ are the exponents of the abelian groups $H^R_i(c^n_{\bm{\lambda}};k)$, and $H^Q_i(c^\dt_{\bm{\lambda}})$, respectively, for $n \gg 0$.\\
\end{enumerate}
\end{corollary}

\subsection{Some $\FI$-module invariants of links}

In this section we use the $\FI$-quandles $c^\dt_{\bm{\lambda}}$ to construct invariants of links. Recall the quandle counting invariant $\chi_L(X) := |\Hom(\K(L),X)|$, where $X$ is a finite quandle, and $\K(L)$ is the fundamental quandle of $L$. In the case wherein $X = c^m_{\bm{\lambda}}$, one immediately observes that $\Hom(\K(L),X)$ is not only a set, it also caries the action of $\Sn_m$. Keeping track of this action will allow us to prove some non-trivial facts about $\chi_L(c^m_{\bm{\lambda}})$.\\

\begin{definition}
Let $\bm{\lambda} \subseteq \Par$ be a finite set of primitive partitions, and let $L$ be an oriented link. For each $n \geq 0$, let $V^{L,\bm{\lambda}}$ denote the rational $\FI$-module defined by,
\[
V^{L,\bm{\lambda}}_n := \Q[\Hom(\K(L),c^n_{\bm{\lambda}}],
\]
the rational vector space with basis vectors indexed by the set $\Hom(\K(L),c^n_{\bm{\lambda}})$.
\end{definition}

Our first goal will be to show that this module is finitely generated. To this end, we have the following.\\

\begin{proposition}\label{fginv1}
Let $\bm{\lambda} \subseteq \Par$ be a finite set of primitive partitions, and let $L$ be an oriented link. Then the $\FI$-module $V^{L,\bm{\lambda}}$ is finitely generated
\end{proposition}

\begin{proof}
Assume that $L$ admits a diagram with $l$ arcs. Then the $\FI$-module $V^{L,\bm{\lambda}}$ is clearly a submodule of the $\FI$-module $C_l(c^\dt_{\bm{\lambda}};\Q)$ (see the proof of Theorem \ref{symhomfg}). The Noetherian property implies that $V^{L,\bm{\lambda}}$ is finitely generated.\\
\end{proof}

Note that it was shown in the above proof that $V^{L,\bm{\lambda}}$ is $(l\cdot|\bm{\lambda}|)$-small. Indeed, this follows from the work in the proof of Theorem \ref{symhomfg}. We will actually prove that $V^{L,\bm{\lambda}}$ is generated in degree $\leq l\cdot|\bm{\lambda}|$. To accomplish this, we will actually need a few more technical lemmas from the representation theory of categories. To start, we have the following.\\

\begin{definition}
The category $\FI_{\sharp}$ is that whose objects are the sets $[n] = \{1,\ldots, n\}$, and whose morphisms are partially defined injections. That is, an element $f \in \Hom_{\FI_{\sharp}}([n],[m])$ is an injection of sets $f:A \hookrightarrow [m]$, where $A \subseteq [n]$. There is a natural inclusion of categories $\FI \subseteq \FI_{\sharp}$, inducing a natural forgetful map between $\FI_{\sharp}$-modules and $\FI$-modules. Therefore, any $\FI_{\sharp}$-module can be considered as an $\FI$-module.\\
\end{definition}

In their work \cite{CEF}, Church, Ellenberg, and Farb famous proved the following structural theorem for $\FI_{\sharp}$-modules. It has since been expanded upon and worked into much deeper and abstract frameworks (see, for instance, \cite{LS}).\\

\begin{theorem}[Church, Ellenberg, \& Farb, \cite{CEF}]\label{sharpstructure}
Let $V$ be a finitely generated $\FI_{\sharp}$-module over a commutative ring $\Q$. Then,
\begin{enumerate}
\item $V$ is free and finitely generated when considered as an $\FI$-module (see Example \ref{free}).
\item If we write $V = \bigoplus_i M(W_i)$, where $W_i$ is a $\Sn_i$-representation, and $V'$ is an $\FI_{\sharp}$ submodule of $V$, then $V' = \bigoplus_i M(W'_i)$ where $W'_i$ is a subrepresentation of $W_i$. In particular, if $V$ is generated in degree $\leq d$, then the same is true of $V'$.
\item There exists a polynomial $p_V \in \Q[T]$ such that, for all $n \geq 0$, $p_V(n) = \dim_\Q V_n$. The degree of the polynomial $p_V$ is precisely the generating degree of $V$.\\
\end{enumerate}
\end{theorem}

This structure theorem is the main piece we need to prove the following.\\

\begin{theorem}\label{fiinv}
Let $\bm{\lambda} \subseteq \Par$ be a finite set of primitive partitions, and let $L$ be an oriented link. If $L$ admits a link diagram with $l$ arcs, then the $\FI$-module $V^{L,\bm{\lambda}}$ is a free-module generated in degree $\leq l \cdot |\bm{\lambda}|$.\\
\end{theorem}

\begin{proof}
To begin, we will show that the $\FI$-module $C_l(c_{\bm{\lambda}}^\dt;\Q)$ (as defined in the proof of Theorem \ref{symhomfg}) can be extended to be an $\FI_{\sharp}$-module.\\

Let $f:A \rightarrow [m]$ be a morphism $f \in \Hom_{\FI_{\sharp}}([n],[m])$. We need to define
\[
f_\as:C_l(c_{\bm{\lambda}}^n;\Q) \rightarrow C_l(c_{\bm{\lambda}}^m;\Q),
\]
in such a way that it agrees with the normal $\FI$-module structure if $A = [n]$. Ideally, we would like to set the action of $f_\as$ on basis vectors as,
\[
f_\as(\sigma_1,\ldots,\sigma_l) = (f_\as\sigma_1,\ldots,f_\as\sigma_l),
\]
where
\[
f_\as\sigma = \begin{cases} f \circ \sigma \circ f^{-1}(i) &\text{ if $i \in \im(f)$}\\ i &\text{ otherwise.}\end{cases}
\]
However, if $A$ is a proper subset of $[n]$, there is no guarantee that $f_\as\sigma$ will have a cycle structure that places it in $c_{\bm{\lambda}}^m$. This can be easily fixed by just annihilating the entire basis vector in this case. Namely,
\[
f_\as(\sigma_1,\ldots,\sigma_l) = \begin{cases} (f_\as\sigma_1,\ldots,f_\as\sigma_l) &\text{ if $f_\as\sigma_j$ has the same cycle structure as $\sigma_j$, up to 1-cycles, for all $j$.}\\ 0 &\text{ otherwise.}\end{cases}
\]
It is easily checked that this assignment is well defined, and gives our desired $\FI_{\sharp}$-module structure.\\

To finish, we note that $V^{L,\bm{\lambda}}$ can be viewed as a submodule of $C_l(c^\dt_{\bm{\lambda}};\Q)$, spanned by tuples $(\sigma_1,\ldots,\sigma_l)$ which satisfy certain conjugacy relations between their members. It is clear that the above $\FI_{\sharp}$-structure will preserve such tuples. In particular, $V^{L,\bm{\lambda}}$ is an $\FI_{\sharp}$-submodule of $C_l(c^\dt_{\bm{\lambda}};\Q)$. The proof of Proposition \ref{fginv1} implies that $C_l(c^\dt_{\bm{\lambda}};\Q)$ is finitely generated in degree $\leq l\cdot|\bm{\lambda}|$. Theorem \ref{sharpstructure} implies that the same is true about $V^{L,\bm{\lambda}}$.\\
\end{proof}

\begin{example}
Let's compute the module $V^{L,\bm{\lambda}}$ for some small examples. First, let $L$ be the trefoil knot, oriented as in Figure \ref{arcs}, and let $\bm{\lambda} = \{(2)\}$. In other words, $c_{\bm{\lambda}}^n$ is the quandle of transpositions of $\Sn_n$. Theorem \ref{fiinv} implies that the entire module $V^{L,\bm{\lambda}}$ is determined by $V^{L,\bm{\lambda}}_n$, where $n = 0,\ldots,6$. In fact, we claim that this module is generated in degree $\leq 3$.\\

A basis vector of $V^{L,\bm{\lambda}}_n$ is the same as a triple $(\tau_x,\tau_y,\tau_z)$ of transpositions of $\Sn_n$, such that
\begin{eqnarray*}
\tau_x &=& \tau_z\tau_y\tau_z\\
\tau_y &=& \tau_x\tau_z\tau_x\\
\tau_z &=& \tau_y\tau_x\tau_y
\end{eqnarray*}
Solving these equations reveals that $(\tau_x\tau_y)^3 = 1$ and $\tau_z = \tau_y\tau_x\tau_y$. Therefore, either $\tau_x = \tau_y = \tau_z$, or $\tau_x = (ij), \tau_y = (jk),$ and $\tau_z = (ik)$ for some distinct $i,j,k$. It follows that
\[
\dim_\Q V^{L,\bm{\lambda}}_n = \binom{n}{2} + 2\binom{n}{2}(n-2).
\]
Theorem \ref{sharpstructure} implies that the module $V^{L,\bm{\lambda}}$ is generated in degree $\leq 3$. In fact, one can see that
\[
V^{L,\bm{\lambda}} = M(\text{triv}_2) \oplus M(W_3)
\]
where triv$_2$ is the trivial representation of $\Sn_2$, and $W_3$ is the permutation representation associated to the action of $\Sn_3$ on the set $\{((12),(23)),((12),(13)),((13),(12)),((13),(23)),((23),(12)),((23),(13))\}$.\\

The previous example illustrates that $V^{L,\bm{\lambda}}$ can be generated in degree $< l \cdot |\bm{\lambda}|$. However, this bound is also seen to be sharp in some examples. Let $L$ denote a pair of linked unknots, each of which is oriented clockwise. This is sometimes referred to as the Hopf link $2^2_1$. Then we have,
\[
\K(L) = \langle x,y \mid x \rhd y = x, y \rhd x = y\rangle.
\]
Keeping the same $\bm{\lambda}$ as the previous example, a basis vector of $V^{L,\bm{\lambda}}_n$ is a pair of commuting transpositions $(\tau_x,\tau_y)$. Transpositions only commute if they are identical, or if they are disjoint. Therefore,
\[
\dim_\Q V^{L,\bm{\lambda}}_n = \binom{n}{2} + \binom{n}{2}\binom{n-2}{2}.
\]
This is a polynomial of degree $4 = 2 \cdot 2 = l \cdot |\bm{\lambda}|$.\\

It therefore becomes an interesting question to ask whether there are more sophisticated methods by which one can compute the generating degree of $V^{L,\bm{\lambda}}$.\\
\end{example}

\begin{remark}
If $\bm{\lambda} = \{(2)\}$ then it is easily seen that $c^3_{\bm{\lambda}} = \Z/3\Z$, the dihedral quandle on three elements. In particular, the link invariant $V^{L,\bm{\lambda}}$ encodes the very classical Fox tricolor invariant of links. The computation above reveals that the Hopf link and the unknot have distinct $V^{L,\bm{\lambda}}$, despite having the same tricoloring number.\\

It is interesting to ask what other classical invariants (if any) are encoded by $V^{L,\bm{\lambda}}$.\\
\end{remark}

\begin{remark}
If $L$ is a knot, then it is easily seen that
\[
V^{L,\bm{\lambda}} = \oplus_{\lambda \in \bm{\lambda}} V^{L,\{\lambda\}}
\]
To see this, note that the relations in $\K(L)$ can be written $\alpha_i = \alpha_j \rhd \alpha_k$, where each $\alpha$ corresponds to an arc in a diagram of $L$. Every arc in this diagram will appear on the left hand side of this relation. Moreover, because knots only have one component by definition, given any arcs $\alpha,\beta$ there exists sequence of arcs $\alpha_1,\ldots,\alpha_c$ such that
\[
\alpha = \alpha_1 \rhd \alpha_2, \text{  } \alpha_1 = \alpha_3 \rhd \alpha_4, \text{  } \ldots \text{  }, \beta = \alpha_{c-1} \rhd \alpha_{c}
\]
In particular, if $L$ is a knot, then the image of any quandle homomorphism $\phi:\K(L) \rightarrow c^{n}_{\bm{\lambda}}$ must land entirely within a single conjugacy class.\\

All of the above implies that, for knots, it suffices to understand $V^{L,\{\lambda\}}$ for all primitive partitions $\lambda$.\\
\end{remark}

While we defined an infinite family of $\FI$-module invariants to each link, it is certainly the case that these modules can become very difficult to compute. One should therefore note that $\FI$-modules themselves have invariants of varying levels of computability, each of which can now be thought of as an invariant of the link. These include:

\begin{itemize}
\item The polynomial describing the dimension of $V^{L,\bm{\lambda}}_n$,
\item If $\bm{\lambda} = \{\lambda\}$ is a single partition, say of an integer $m$, then it can be shown that the polynomial describing the dimension of $V^{L,\bm{\lambda}}_n$ is divisible by the polynomial $|c^n_{\lambda}| = \binom{n}{m}|c^m_{\lambda}$. The quotient of the dimension polynomial by $|c^n_{\lambda}|$ can be thought of as a kind of normalization of the original dimension polynomial with respect to the unknot,
\item the generating degree of $V^{L,\bm{\lambda}}$,
\item writing $V^{L,\bm{\lambda}} = \bigoplus_i M(W_i)$, the multiplicities of the irreducible representations constituting the $W_i$,
\item For $n \gg 0$, the multiplicity of the trivial representation in $V^{L,\bm{\lambda}}_n$ is constant \cite{CEF}. This constant value is an invariant of the module.\\
\end{itemize}

\section{Some $\VIC(q)$-quandles}\label{vic}

For the remainder of this section, we will fix a finite field $F$ of order $q$. Note that many of the arguments in this section are very similar to those given in the previous section. For this reason we will often skip details in certain arguments.\\

\subsection{Definitions}
Just as before, we begin by reviewing the conjugacy classes of $GL_n(q)$. The picture here is analogous, although not nearly as simple.\\

Given a matrix $T \in GL_n(q)$, we obtain an action of $F[x]$ on $F^n$ by multiplication by $T$. The structure theorem for modules over a PID then implies
\[
F^n = \bigoplus_i F[x]/(f_i)^{e_i}
\]
where $e_i \geq 0$ is an integer, and $f_i$ is an irreducible polynomial. This decomposition uniquely determines the conjugacy class of the matrix $T$. We can encode this information in the following way.\\

\begin{definition}
Let $\Par$ be the set of partitions, and let Poly$(q)$ denote the set of monic irreducible polynomials over $F[x]$, not including the polynomial $x$. We will write $\Phi$ to denote a function of sets,
\[
\Phi:\text{Poly}(q) \rightarrow \Par,
\]
such that,
\[
\sum_{f \in \Phi}\sum_{i \geq 0} \deg(f)\cdot\Phi(f)_i = n.
\]
Any function $\Phi$, as above, corresponds to a conjugacy class of $GL_n(q)$ by setting
\[
F^n = \bigoplus_{f \in \text{Poly}(q)}\bigoplus_{i \geq 0} F[x]/(f)^{\Phi(f)_i}
\]
We will write $c_{\Phi}$ for this conjugacy class.\\

We say $c_{\Phi}$ is \textbf{primitive} if $\Phi(x-1)$ is primitive. We say that $\Phi$ is \textbf{primitive} if $c_\Phi$ is primitive. Given a primitive conjugacy class $c_\Phi$ of $GL_n(q)$, and an integer $m \geq n$, we define a new conjugacy class of $GL_{m}(q)$ by defining
\[
\Phi^{m-n}(f) := \begin{cases} \Phi(f) &\text{ if $f \neq x-1$}\\ \Phi(x-1)^{m-n} &\text{ otherwise.}\end{cases}
\]
and setting
\[
c^m_{\Phi} := c_{\Phi^{m-n}}.
\]
Note that if $m < n$ then $c^m_\Phi$ is the empty set, by definition.\\
\end{definition}

If $c_\Phi$ is a conjugacy class of $GL_n(q)$, then we think of $c_{\Phi}^m$ as the conjugacy class of $GL_m(q)$ obtained from $c_\Phi$ by adding $m-n$ linearly independent eigenvectors for 1. This is analogous to the symmetric group case, which involved adding 1-cycles to our cycle decomposition.\\

\begin{definition}
Let $\bm{\Phi}$ denote a finite set of primitive maps $\Phi:\text{Poly}(q) \rightarrow \Par$. Then for each $n$ we set $c^n_{\bm{\Phi}}$ to be the quandle
\[
\cup_{\Phi \in \bm{\Phi}} c^n_\Phi \subseteq GL_n(q)
\]
Let $(f,W):F^m \rightarrow F^n$ is a $\VIC(q)$ morphism, so that any element of $F^n$ can be written uniquely as $v = f(v_1) + v_2$ where $v_1 \in F^m$ and $v_2 \in W$. If $T \in c^m_{\bm{\Phi}}$, then we define $(f,W)_\as T \in GL_n(q)$ to be the matrix
\[
((f,W)_\as T)(v) = f \circ T(v_1) + v_2
\]
It is easily checked that $(f,W)_\as T$ is an element of $c^n_{\bm{\Phi}}$, and that this extends the conjugacy action of $GL_n(q)$ on $c^n_{\bm{\Phi}}$. In particular, the assignment $n \mapsto c^n_{\bm{\Phi}}$ turns $c^\dt_{\bm{\Phi}}$ into a $\VIC(q)$-quandle.\\
\end{definition}

\begin{remark}
The definition of $(f,W)_\as$ is the main inspiration for the category $\VIC(q)$. Namely, by specifying the compliment of the image of $f$, one naturally obtains a way to use $f$ to map between $GL_m(q)$ and $GL_n(q)$.\\
\end{remark}

\subsection{Proving that the homology is finitely generated}

In this section, we will prove that the quandle and rack homology groups of $c^\dt_{\bm{\Phi}}$ are finitely generated $\VIC(q)$-modules over any Noetherian ring. For the remainder of this section, we fix a finite set $\bm{\Phi}$ of primitive functions $\Phi:\text{Poly}(q) \rightarrow \Par$.\\

If $c_\Phi$ is a conjugacy class of $GL_m(q)$, then we write $|\Phi| = m$. If $\bm{\Phi}$ is as above, we write $|\bm{\Phi}| = \max_{\Phi \in \bm{\Phi}}|\Phi|$.\\

\begin{theorem}\label{symlinfg}
Let $\bm{\Phi}$ be a finite set of primitive functions $\Phi:\text{Poly}(q) \rightarrow \Par$, let $i \geq 0$ be a fixed integer, and let $k$ be a commutative Noetherian ring. Then the $\VIC(q)$-modules $H^R_i(c^\dt_{\bm{\Phi}};k)$ and $H^Q_i(c^\dt_{\bm{\Phi}};k)$ are finitely generated. Moreover, both $H^R_i(c^\dt_{\bm{\Phi}};k)$ and $H^Q_i(c^\dt_{\bm{\Phi}})$ are $(i\cdot |\bm{\Phi}|)$-small.
\end{theorem}

\begin{proof}
The proof of this theorem is very similar to the proof for Theorem \ref{symhomfg}. In this case, we must show that if $n > i\cdot |\bm{\Phi}|$ then any $i$-tuple $(T_1,\ldots,T_i)$ of elements in $c^n_{\bm{\Phi}}$ has a common eigenvector for 1.\\

Let $E_j$ denote the 1-eigenspace for $T_j$. By how $c^n_{\bm{\Phi}}$ is defined, we know that $\dim_F E_j \geq n-|\bm{\Phi}|$ for each $j$. Using standard dimension formulas we know that
\[
n \geq \dim_F (E_1 + E_2) = \dim_F E_1 + \dim_F E_2 - \dim_F (E_1 \cap E_2) \geq 2n - 2|\bm{\Phi}| - \dim_F (E_1 \cap E_2).
\]
Therefore,
\[
\dim_F(E_1 \cap E_2) \geq n - 2|\bm{\Phi}|.
\]
Proceeding by induction we conclude,
\[
\dim_F(E_1 \cap E_2 \cap \ldots \cap E_i) \geq n - i|\bm{\Phi}|>0.
\]
This concludes the proof.\\
\end{proof}

Once again we can use this result, along with Theorem \ref{maincor} To conclude non-trivial facts about these homology groups.\\

\begin{corollary}\label{lincor}
Let $\bm{\Phi}$ be a finite set of primitive functions $\Phi:\text{Poly}(q) \rightarrow \Par$, let $i \geq 0$ be a fixed integer, and let $k$ be a commutative Noetherian ring. For each $n$, let $\ai_{R,i,n},\ai_{Q,i,n} \subseteq k$ be the ideals generated by non-zero-divisors which annihilate $H^R_i(c^n_{\bm{\Phi}};k)$ and $H^Q_i(c^n_{\bm{\Phi}};k)$, respectively. Then for $n \gg 0$, $\ai_{R,i,n}$ and $\ai_{Q,i,n}$ are independent of $n$. In particular, if $k = \Z$, then there exist integers $e_{R,i}$ and $e_{Q,i}$, independent of $n$, such that $e_{R,i}$ and $e_{Q,i}$ are the exponents of the abelian groups $H^R_i(c^n_{\bm{\Phi}};k)$, and $H^Q_i(c^\dt_{\bm{\Phi}})$, respectively, for $n \gg 0$.\\
\end{corollary}

\subsection{Some $\VIC(q)$-module invariants of links}

To conclude, we consider $\VIC(q)$-modules associated to links. Unfortunately, these invariants will not prove to be as easily computable as in the $\FI$ case. This is due to the (relative) lack of structure theorems for $\VIC(q)$-modules as compared to $\FI$-modules. In any case, we have the following.\\

\begin{definition}
Let $\bm{\Phi}$ be a finite collection of primitive functions $\Phi:\text{Poly}(q) \rightarrow \Par$, and let $L$ be an oriented link. Then we define a rational $\VIC(q)$-module $V^{L,\bm{\Phi}}$ by the assignments
\[
V^{L,\bm{\Phi}}_n := \Q[\Hom(\K(L),c^n_{\bm{\Phi}}],
\]
the $\Q$ vector space with basis indexed by the set $\Hom(\K(L),c^n_{\bm{\Phi}})$.\\
\end{definition}

In the case of $\FI$-modules, we were able to prove strong theorems about the modules $V^{L,\bm{\lambda}}$ by appealing to the extra $\FI_{\sharp}$-module structure. There are no known analogous constructions for $\VIC(q)$-modules. In fact, structure theorems for $\VIC(q)$-modules, and representations of related categories, are still a very active field of research (see \cite{PS,GW,SS}, for example).\\

The following theorem is proven in almost the exact same way as Proposition \ref{fginv1}.\\

\begin{theorem}\label{fginv2}
Let $\bm{\Phi}$ be a finite collection of primitive functions $\Phi:\text{Poly}(q) \rightarrow \Par$, and let $L$ be an oriented link. If $L$ admits a diagram with $l$ arcs, then the $\VIC(q)$-module $V^{L,\bm{\Phi}}$ is $(l\cdot |\bm{\Phi}|)$-small. In particular, $V^{L,\bm{\Phi}}$ is finitely generated, and there exists a polynomial $p_{L,\bm{\Phi}}(T) \in \Q[T]$ of degree $\leq l\cdot |\bm{\Phi}|$ such that for all $n \gg 0$,
\[
\dim_\Q V^{L,\bm{\Phi}}_n = p_{L,\bm{\Phi}}(q^n)
\]
\end{theorem}

\end{document}